\documentclass{amsart}
\usepackage{amsmath,amsthm,amssymb}
\usepackage[colorlinks]{hyperref}
\newtheorem{thm}{Theorem}
\newtheorem*{prop*}{Proposition}
\newtheorem*{lem*}{Lemma}
\newtheorem*{rem*}{Remark}

\begin{document}
\title{Equivalent Bergman Spaces With Inequivalent Weights}
\author{Blake J. Boudreaux}
\date{\today}
\address{Department of Mathematics, Texas A\&M University, College Station TX 77843-3368}
\email{bboudreaux28@math.tamu.edu}

\begin{abstract}
We give a proof that every space of weighted square-integrable
holomorphic functions admits an equivalent weight whose Bergman
kernel has zeroes. Here the weights are equivalent in the sense that they determine the same space of holomorphic functions. Additionally, a family of radial weights in $L^1(\mathbb{C})$ whose associated Bergman kernels have infinitely many zeroes is exhibited.
\end{abstract}

\maketitle

\section{Introduction}
Since the paper of L. Qi-Keng \cite{Qi-Keng1966}, there has been interest in constructing domains in $\mathbb{C}^N$ whose Bergman kernel has zeroes. After the observation that weighted Bergman kernels correspond to unweighted Bergman kernels in higher dimension, it is natural to consider the same questions involving weighted Bergman kernels.

A positive measurable function $\mu$ defined on a domain $D\subset\mathbb{C}^N$ is called a \textit{weight}. (We merely require a weight to be measurable; some authors require weights to additionally be integrable.) To every weight $\mu$ on $D$ corresponds a Hilbert space $L^2(D,\mu)$ of measurable functions determined by the inner product
\[
\left\langle f,g\right\rangle_{\mu} :=\int_{D}f(\zeta)\overline{g(\zeta)}\mu(\zeta)dA(\zeta).
\]
Let $L^2_H(D,\mu)$ denote the subspace of $L^2(D,\mu)$ consisting of those functions that are also holomorphic.

We are interested in weights that determine a space on which a weighted Bergman kernel can be defined. We call a weight $\mu$ \textit{admissible} if for each $z\in D$ the evaluation functional $E_z:f\mapsto f(z)$ is continuous on $L^2_H(D,\mu)$, and if $L^2_H(D,\mu)$ is a closed subspace of $L^2(D,\mu)$. In the case that $\mu$ is admissible, for each $z\in D$ the Riesz representation theorem provides a unique $B_z^{D,\mu}(\,\cdot\,)\in L^2_H(D,\mu)$ such that
\[
f(z)=\left\langle f,B_z^{D,\mu}\right\rangle_{\mu}.
\]
It is common to write $K_{D,\mu}(z,\zeta)=B^{D,\mu}_z(\zeta)$ and view it as a function on $D\times D$.  We call $K_{D,\mu}$ the weighted Bergman kernel of $D$ (with respect to the weight $\mu$). We typically write $K_{\mu}$ in place of $K_{D,\mu}$ if the domain is clear from context. As in the unweighted case \cite{Krantz2001,Range1986}, the kernel $K_{\mu}(z,\zeta)$ possesses the following properties:
\begin{enumerate}
\item[(i)] $K_{\mu}(z,\zeta)$ is holomorphic in $z$ and conjugate-holomorphic in $\zeta$,
\item[(ii)] $K_{\mu}(z,\zeta)=\sum_{j}\varphi_{j}(z)\overline{\varphi_{j}(\zeta)}$ for any complete orthonormal system $\{\varphi_j\}$ of $L^2_H(D,\mu)$, with convergence uniform on compact sets of $D\times D$,
\item[(iii)] and $\overline{K_{\mu}(z,\zeta)}=K_{\mu}(\zeta,z)$.
\end{enumerate}
A comprehensive reference on the theory of admissible weights is the paper, ``On the Dependence of the Reproducing Kernel on the Weight of Integration''\cite{PW1990}.

Following the convention of A. Per\"al\"a \cite{Perälä2017}, we say that two admissible weights $\mu_1$ and $\mu_2$ are \textit{equivalent}, or $\mu_1\sim\mu_2$, if $L^2_H(D,\mu_1)=L^2_H(D,\mu_2)$ as sets. For example, if $g$ is a positive measurable function on $D$ with the property that $\text{ess inf}_{z\in D}g(z)>0$ and $\text{ess sup}_{z\in D}g(z)<\infty$, then $g\cdot \mu\sim\mu$ for any weight $\mu$ on $D$.

The purpose of this note is to answer two questions \cite{Perälä2017}. The firsts asks if every space $L^2_H(D,\mu)$ can be equipped with an equivalent weight $\mu^*$ so that the Bergman kernel $K_{D,\mu^*}$ has zeroes. The second asks if there exists a radial weight $W$ on $\mathbb{C}$ such that the kernel $K_{\mathbb{C},W}(\,\cdot\, ,z)$ has infinitely many zeros for a fixed $z\in D$. We answer both questions in the affirmative.

The plan of attack to answering the first question above is best illustrated by considering the case when $D\subset\mathbb{C}^N$ is a bounded set containing zero and $\mu$ is continuous. In this situation, the weight $\nu(z):=\mu(z)/\|z\|^{2N}$ is not locally integrable at zero. Indeed, by continuity $\mu$ is uniformly bounded away from zero in a sufficiently small neighborhood of $z=0$. It follows that every member of $L^2_H(D,\nu)$\textemdash in particular $K_{\nu}(\,\cdot\,,\zeta)$ for each $\zeta\in D$\textemdash vanishes at zero. Moreover, for each $n\in\mathbb{N}$ the weight $\nu_n(z):=\min(n,\|z\|^{-2N})\cdot\mu(z)$ is equivalent to $\mu(z)$: each $\nu_n$ is simply the product of $\mu$ and a bounded function that is uniformly bounded away from zero (recall we are assuming $D$ is bounded). Since $\nu_n$ increases to $\nu$ as $n\to\infty$, we may apply a weighted generalization of the Ramadanov theorem \cite{PWW2016} to see that $K_{\nu_n}\to K_{\nu}$ uniformly on compact subsets of $D\times D$. Regarding $\zeta\in D$ as fixed, the function $K_{\nu}(\,\cdot\,,\zeta)$ vanishes at zero, so a variant of Hurwitz's theorem shows that for large $n$ we have $K_{\nu_n}(0,\zeta)=0$, completing the proof. The general case requires a more delicate approach, as a general measurable function may have extremely pathlogical behavior near every point in its domain, but the main idea remains the same.

The approach to answering the second question involves carefully choosing the weight $W$ so that for each fixed nonzero $w$ in the plane the kernel $K_{\mathbb{C},W}(\,\cdot\,,w)$ is an entire function of finite but non-integer order, which by a consequence of the Hadamard factorization theorem has infinitely many zeros. As a special case of our construction, we exhibit a weight with kernel $\cos(i\sqrt{z\bar{w}})$; this is a function satisfying the necessary criteria by elementary means.

I would like to thank my advisor, Dr. Harold Boas, for bringing these questions to my attention, as well as providing direction on how one might approach them. I would also like to thank the referee, whose insightful comments transformed the third section of this note from a single example to a large family.

\section{The Bergman Kernels of Equivalent Weights}
Let us first show that our definition of an admissible weight is consistent with another standard definition \cite{PW1990}.
\begin{prop*}
Let $\mu$ be a weight on a domain $D\subset\mathbb{C}^N$. Then $\mu$ is admissible if and only if the norm of the point evaluation functional $E_z:f\mapsto f(z)$ is locally bounded (if thought of as a function on $D$).
\end{prop*}
\begin{proof}
Suppose that $\mu$ is admissible. Fix $z\in D$ and let $V_z$ be an open set containing $z$ with $\overline{V_z}\subset D$. Since
\[
\sup_{w\in V_z}|E_w(f)|=\sup_{w\in V_z}|f(w)|<\infty
\]
for every $f\in L^2_H(D,\mu)$, an application of the uniform boundedness principle to the family $\{E_w\,:\,w\in V_z\}$ of continuous linear functionals shows that
\[
\sup_{w\in V_z}\|E_w\|<\infty.
\]
The converse follows from known results (\cite{PW1990} Proposition 2.1).
\end{proof}
Before proceeding we require a lemma. It allows us in many cases to assume the given weight is bounded below by a more well behaved weight.
\begin{lem*}
Let $\mu_1$ be an admissible weight on $D$, and suppose that $\mu_1$ is integrable on a bounded open neighborhood $U$ of some point $z_0\in D$ with $\overline{U}\subset D$. Let $\mu_2$ be a weight on $U$. Then the weight $\tilde{\mu}_1$ defined by
\[\tilde{\mu}_1(z):= \begin{cases} 
\max(\mu_1(z),\mu_2(z))\,& \text{ if }z\in U\\
      \mu_1(z) & \text{ if } z\in D\setminus U \end{cases}
\]
is an admissible weight with $L^2_H(D,\tilde{\mu}_1)\subset L^2_H(D,\mu_1)$ and continuous inclusion, such that $\mu_2\leq \tilde{\mu}_1$ on $U$. Furthermore, if $\mu_2$ is integrable over $U$ as well, then $L^2_H(D,\mu_1)=L^2_H(D,\tilde{\mu}_1)$ and $\tilde{\mu}_1$ determines an equivalent norm to $\mu_1$.
\end{lem*}
\begin{proof}
Observe that $\|f\|_{\mu_1}\leq\|f\|_{\tilde{\mu}_1}$ for every measurable function $f$. Therefore $L^2(D,\tilde{\mu}_1)\subset L^2(D,\mu_1)$ with continuous inclusion. It follows that $\|E_{z}\|_{\tilde{\mu}_1}\leq\|E_{z}\|_{\mu_1}$ for every $z\in D$, and hence $\tilde{\mu}_1$ is an admissible weight.

Now suppose that $\mu_2$ is integrable over $U$ as well. By applying the uniform boundedness principle to the family of continuous functionals given by evaluation at each point of $U$, we may find a $C>0$ such that
\begin{align*}
\|f\|_{\tilde{\mu}_1}^2 &= \int_{U}|f(\zeta)|^2\max(\mu_1(\zeta),\mu_2(\zeta))dA(\zeta) +\int_{D\setminus U}|f(\zeta)|^2\mu_1(\zeta) dA(\zeta)\\
&\leq\sup_{z\in U}|f(z)|^2\cdot\int_{U}\max(\mu_1(\zeta),\mu_2(\zeta))dA(\zeta) +\|f\|_{\mu_1}^2\\
&\leq C\cdot\left(\int_{U}\max(\mu_1(\zeta),\mu_2(\zeta))dA(\zeta)\right)\cdot \|f\|_{\mu_1}^2+\|f\|^2_{\mu_1}\\
&\leq \left[C\left(\int_{U}\max(\mu_1(\zeta),\mu_2(\zeta))dA(\zeta)\right)+1\right]\|f\|^2_{\mu_1}
\end{align*}
holds for each $f\in L^2_H(D,\mu_1)$. Observe that $\max(\mu_1,\mu_2)$ is integrable over $U$ since both $\mu_1$ and $\mu_2$ are. 
\end{proof}
By setting $\mu_2\equiv 1$ above, we have the immediate corollary that $\tilde{\mu}_1$ is an equivalent weight to $\mu_1$ with equivalent norm having the property that $1\leq \tilde{\mu}_1$ on $U$.

We also require a weighted generalization of the Ramadanov theorem \cite{PWW2016}, whose statement is included for the convenience of the reader.
\begin{thm}[Weighted generalization of the Ramadanov theorem]\label{Ramadanov}
Let $\{D_i\}_{i=1}^{\infty}$ be a sequence of domains in $\mathbb{C}^N$ and set $D:=\bigcup_{j}D_j$. Let $\mu$ be an admissible weight on $D$, and $\mu_k$ be an admissible weight on $D_k$ for each $k$. Extend $\mu_k$ by $\mu$ on $D$. Assume moreover that
\begin{itemize}
\item[a)] For any $n\in\mathbb{N}$ there is $N=N(n)$, such that $D_n\subset D_m$ and $\mu_n(z)\leq\mu_m(z)\leq\mu(z)$ for $m\geq N(n)$, $z\in D_n$.
\item[b)] $\mu_k\xrightarrow[k\to\infty]{}\mu$ pointwise almost everywhere on $D$.
\end{itemize}
Then
\[
\lim_{k\to\infty}K_{D_k,\mu_k}=K_{D,\mu}
\]
locally uniformly on $D\times D$.
\end{thm}
We now have all the necessary tools to prove the main result of this section.
\begin{thm}\label{Main}
Let $\mu$ be an admissible weight on a domain $D\subset\mathbb{C}^N$. Then there exists an admissible weight $\mu^*$, with $\mu\sim\mu^*$, such that $K_{D,\mu^*}$ has zeroes.
\end{thm}
\begin{proof} We assume that $L_H^2(D,\mu)\neq \{0\}$, otherwise the Bergman kernel vanishes identically. By translating if necessary, we may assume that $0\in D$. If $\mu$ is not integrable in any neighborhood of $z=0$, then every $f\in L^2_H(D,\mu)$ must satisfy $f(0)=0$; in particular this implies that $K_{\mu}(0,\zeta)=0$ for each $\zeta\in D$. Therefore we may assume that $\mu$ is integrable on some neighborhood $U$ of $z=0$. By shrinking $U$ if necessary, we may assume that $U$ is bounded with $\overline{U}\subset D$. The Lemma now allows us to assume that $1\leq \mu$ on $U$.

Set $g(z)=\max \left(1,1/\|z\|^{2N}\right)$ and consider
\[
\nu(z)=g(z)\mu(z).
\]
Since $1\leq g(z)$ everywhere, we have $\|f\|^2_{\mu}\leq\|f\|^2_{\nu}$. This shows that the inclusion $L^2(D,\nu)\subset L^2(D,\mu)$ is continuous, implying that $\nu$ is an admissible weight (as in the proof of the Lemma). Since we assume that $1\leq\mu(z)$ on $U$, $\nu$ is not integrable in any neighborhood of $z=0$ and hence $K_{D,\nu}(\,\cdot\,,w)=0$ for each $w\in D$.

The function $\min(n,g(z))$ is bounded above and uniformly bounded away from zero on $D$, so
\[
\nu_n(z)=\min\left(n,g(z)\right)\cdot \mu(z)
\]
is equivalent to $\mu$ for each $n\in\mathbb{N}$.

Next, we apply Theorem \ref{Ramadanov}.
Since $L^2_H(D,\nu)\neq \{0\}$ (e.g. $z^{\alpha}f\in L^2_H(D,\nu)$ whenever $f\in L^2_H(D,\mu)$ and $\alpha$ is a multindex with $|\alpha|=N$), we may find a $w\in D$ so that $K_{\nu}(z,w)$ is a nontrivial holomorphic function of $z$. We claim that $K_{\nu_{m}}(z,w)$ has zeros for some $m\in\mathbb{N}$. Seeking a contradiction, suppose that $K_{\nu_{n}}(z,w)$ has no zeros for every $n\in\mathbb{N}$. We have chosen $w\in D$ so that $K_{\nu}(z,w)$ is not identically zero, so fix $z_0\in D$ with $K_{\nu}(z_0,w)\neq 0$. Applying Hurwitz's theorem of one complex variable to connected component of $\{\lambda z_0\in D\,:\,\lambda\in\mathbb{C}\}$ containing the origin, we see that $K_{\nu}(\lambda z_0,w)$ has no zeros. However this implies that $K_{\nu}(0,w)\neq 0$, a contradiction to what was shown above: every $f\in L^2_H(D,\nu)$ vanishes at zero. This shows the claim, and setting $\mu^*=\nu_{m}$ completes the proof.
\end{proof}
\textbf{Remark.} Observe that something slightly stronger than the conclusion of the theorem has been shown: that we may\textemdash up to any positive error\textemdash actually prescribe the point at which the zero occurs. Furthermore, by carrying out this construction at finitely many points simultaneously, one can show that an equivalent weight exists whose Bergman kernel has zeroes at finitely many predetermined points up to any positive error.
\section{Radial Weights with Kernel Having Infinitely Many Zeroes in The Plane}
It is known \cite{Perälä2017} that the kernel $B_{\mathbb{D},\mu}(\,\cdot\,,w)$ of an integrable radial weight $\mu$ on the unit disk $\mathbb{D}\subset\mathbb{C}$ cannot have infinitely many zeroes for a fixed $w\in\mathbb{D}$. In this section we show that the analogous result fails when $\mathbb{D}$ is replaced with the complex plane. In fact, we exhibit a family of radial weights $\mathcal{W}\subset L^1(\mathbb{C})$ such that for every $W\in\mathcal{W}$, the associated weighted Bergman kernel $B_{\mathbb{C},W}(\,\cdot\,,w)$ has infinitely many zeroes for each fixed nonzero $w$ in the plane. This is achieved by following a similar construction to that of H. Bommier-Hato, M. Engli\v{s}, and E.-H. Youssfi \cite{BHEEH}.

Given two real and positive parameters $\beta,\gamma$, we may define a holomorphic function $E_{\beta,\gamma}(z)$ by the power series
\begin{equation}
E_{\beta,\gamma}:=\sum_{k=0}^{\infty}\frac{z^k}{\Gamma(\beta k+\gamma)}.
\end{equation}
This is known as the Mittag-Leffler function associated to $\beta$ and $\gamma$. $E_{\beta,\gamma}$ is an entire function with order $1/\beta$ and type 1. A comprehensive treatise on the theory of Mittag-Leffler functions is \textit{Mittag-Leffler functions, Related Topics and Applications} \cite{GAM2014}.
\begin{thm}
Let $\mathcal{W}\subset L^1(\mathbb{C})$ be the family of weights of the form
\[
W(z)=\frac{1}{2\pi}|z|^n\exp(-\alpha|z|^{2m}),
\]
where $n\in (-2,\infty)$, $\alpha,m>0$, and $m\not\in\mathbb{Z}$. Every member $W$ of $\mathcal{W}$ is admissible and induces a kernel $K_{\mathbb{C},W}(\,\cdot\,,w)$ having infinitely many zeroes for each nonzero $w$ in the plane.
\end{thm}
\begin{proof}
Fix $W\in\mathcal{W}$. We first show that $W$ is admissible. By a result of Z. Pasternak-Winiarski (\cite{PW1990(2)} Corollary 3.1), it suffices to show that there exists a $c>0$ such that $W^{-c}$ is locally integrable; setting $c=1/n$ if $n>0$ and $c=1$ otherwise will work. Since $W$ is radial, the monomials are an orthonormal basis for $L_H^2(D,W)$, and the representation
\begin{equation}
B_{W}(z,w)=\sum_{k=0}^{\infty}\frac{1}{W_k}(z\bar{w})^k,
\end{equation}
with $W_k=2\pi\int_{0}^{\infty} r^{2k+1}W(r)dr$, holds. Now
\begin{equation}
W_k=\int_{0}^{\infty}r^{2k+1+n}\exp(-\alpha r^{2m})=\frac{1}{2m}\alpha^{\tfrac{2m-2k-n-3}{2m}}\cdot\Gamma\left(\frac{2k+2+n}{2m}\right).
\end{equation}
Note that $W_1=\|W\|_{L^1}<\infty$, so $W\in L^1(\mathbb{C})$ and hence $\mathcal{W}$ is well defined. Comparing (2) with (3) yields
\begin{align}
B_{W}(z,w)&=2m\alpha^{\tfrac{3-2m+n}{2m}}\sum_{k=0}^{\infty}\alpha^{k/m}\frac{(z\bar{w})^k}{ \Gamma\left(\frac{2k+2+n}{2m}\right)}\\\nonumber &=2m\alpha^{\tfrac{3-2m+n}{2m}}\sum_{k=0}^{\infty}\frac{\big(\alpha^{1/m}(z\bar{w})\big)^k}{ \Gamma\left(\tfrac{k}{m}+\tfrac{2+n}{2m}\right)}.
\end{align}
We may write this in terms of the Mittag-Leffler function (1) as
\[
B_{W}(z,w)=2m\alpha^{\tfrac{3-2m+n}{2m}}E_{\tfrac{1}{m},\tfrac{2+n}{2m}}\big(\alpha^{1/m}(z\bar{w})\big).
\]
Fix a nonzero $w$ in the plane. It follows from (4) that $B_{W}(z,w)$ is an entire function of order $m$. Since $m\not\in\mathbb{Z}$ by construction, it is a  consequence of the Hadamard factorization theorem that the kernel $B_{W}(\,\cdot\,, w)$ has infinitely many zeroes (\cite{C1978} Theorem XI.3.7).
\end{proof}
Observe that setting $m=1/2$, $\alpha=1$, and $n=-1$ in (4) shows
\[
B_{W}(z,w)=\sum_{k=0}^{\infty}\frac{(z\bar{w})^k}{\Gamma(2k+1)}=\sum_{k=0}^{\infty}\frac{(z\bar{w})^k}{(2n)!}=\cos(i\sqrt{z\bar{w}}),
\]
which provides a concrete member of $\mathcal{W}$ that can be shown to satisfy the conclusion of Theorem 3 without having to invoke Hadamard's theorem.\medbreak

\textbf{Remark.} The construction of this family of examples required solving a Stieltjes moment problem whose solution is absolutely continuous with respect to Lebesgue measure. There has been much work done on solving the Stieltjes moment problem \cite{D1989, ST1943}, so it would be interesting to see if one could characterize the entire functions $f$ for which there corresponds an admissible weight $\mu_f$ on $\mathbb{C}$ with $B_{\mathbb{C},\mu_f}(z,w)=f(z\bar{w})$. For instance, it is clear that a necessary condition on such a function $f$ is that its Maclaurin series coefficients be all real and positive.
\bibliography{VBK}

\begin{thebibliography}{10}

\bibitem{BHEEH}
H\'el\`ene Bommier-Hato, Miroslav Engli\v{s}, and El-Hassan Youssfi.
\newblock Bergman-type projections in generalized {F}ock spaces.
\newblock {\em J. Math. Anal. Appl.}, 389(2):1086--1104, 2012.
\newblock
  doi:\href{https://dx.doi.org/10.1016/j.jmaa.2011.12.045}{10.1016/j.jmaa.2011.12.045}.

\bibitem{C1978}
John~B. Conway.
\newblock {\em Functions of {O}ne {C}omplex {V}ariable}, volume~11 of {\em
  Graduate Texts in Mathematics}.
\newblock Springer-Verlag, New York-Berlin, second edition, 1978.
\newblock
  doi:\href{http://dx.doi.org/10.1007/978-1-4612-6313-5}{10.1007/978-1-4612-6313-5}.

\bibitem{D1989}
Antonio~J. Duran.
\newblock The {S}tieltjes moments problem for rapidly decreasing functions.
\newblock {\em Proc. Amer. Math. Soc.}, 107(3):731--741, 1989.
\newblock doi:\href{http://dx.doi.org/10.2307/2048172}{10.2307/2048172}.

\bibitem{GAM2014}
Rudolf Gorenflo, Anatoly~A. Kilbas, Francesco Mainardi, and Sergei~V. Rogosin.
\newblock {\em Mittag-{L}effler {F}unctions, {R}elated {T}opics and
  {A}pplications}.
\newblock Springer Monographs in Mathematics. Springer, Heidelberg, 2014.
\newblock
  doi:\href{https://doi.org/10.1007/978-3-662-43930-2}{10.1007/978-3-662-43930-2}.

\bibitem{Krantz2001}
Steven~G. Krantz.
\newblock {\em Function {T}heory of {S}everal {C}omplex {V}ariables}.
\newblock AMS Chelsea Publishing, Providence, RI, 2001.
\newblock Reprint of the 1992 edition.
  doi:\href{http://dx.doi.org/10.1090/chel/340}{10.1090/chel/340}.

\bibitem{Qi-Keng1966}
Q.-k. Lu.
\newblock On {K}aehler manifolds with constant curvature.
\newblock {\em Chinese Math.--Acta}, 8:283--298, 1966.

\bibitem{PWW2016}
Z.~{Pasternak-Winiarski} and P.~M. {W{\'o}jcicki}.
\newblock {Weighted generalization of the Ramadanov theorem and further
  considerations}.
\newblock {\em ArXiv: 1612.05619}, December 2016.

\bibitem{PW1990}
Zbigniew Pasternak-Winiarski.
\newblock On the dependence of the reproducing kernel on the weight of
  integration.
\newblock {\em J. Funct. Anal.}, 94(1):110--134, 1990.
\newblock
  doi:\href{http://dx.doi.org/10.1016/0022-1236(90)90030-O}{10.1016/0022-1236(90)90030-O}.

\bibitem{PW1990(2)}
Zbigniew Pasternak-Winiarski.
\newblock On weights which admit the reproducing kernel of {B}ergman type.
\newblock {\em Internat. J. Math. Math. Sci.}, 15(1):1--14, 1992.
\newblock
  doi:\href{https://dx.doi.org/10.1155/S0161171292000012}{10.1155/S0161171292000012}.

\bibitem{Perälä2017}
Antti Per{\"a}l{\"a}.
\newblock Vanishing {B}ergman kernels on the disk.
\newblock {\em J. Geom. Anal.}, Jun 2017.
\newblock
  doi:\href{https://doi.org/10.1007/s12220-017-9885-1}{10.1007/s12220-017-9885-1}.

\bibitem{Range1986}
R.~Michael Range.
\newblock {\em Holomorphic {F}unctions and {I}ntegral {R}epresentations in
  {S}everal {C}omplex {V}ariables}, volume 108 of {\em Graduate Texts in
  Mathematics}.
\newblock Springer-Verlag, New York, 1986.
\newblock
  doi:\href{http://dx.doi.org/10.1007/978-1-4757-1918-5}{10.1007/978-1-4757-1918-5}.

\bibitem{ST1943}
J.~A. Shohat and J.~D. Tamarkin.
\newblock {\em The {P}roblem of {M}oments}.
\newblock American Mathematical Society Mathematical surveys, vol. I. American
  Mathematical Society, New York, 1943.

\end{thebibliography}
\bibliographystyle{plain}
\end{document}